\documentclass{amsart}
\usepackage{amsmath, amssymb, epsfig, graphicx, color }
\usepackage[all]{xy}
\newcommand{\Z}{\mathbb Z}

\newcommand{\N}{\mathbb N}

\newcommand{\eps}{\varepsilon}

\usepackage{enumerate,comment}
\newtheorem{thm}{Theorem}[section]

\newtheorem{lem}[thm]{Lemma}
\newtheorem{prop}[thm]{Proposition}
\newtheorem{cor}[thm]{Corollary}

\newtheorem{fact}[thm]{Fact}
\newtheorem*{repeatthm}{Theorem \ref{thm:repeat}}

\newtheorem*{cornonum1}{Corollary \ref{cor:repeat1}}
\newtheorem*{cornonum2}{Corollary \ref{cor:repeat2}}
\newtheorem*{thmnonum}{Theorem}

\begin{document}
\title{Chaos and Indecomposability}
\author[Darji] {Udayan B.~Darji}
\author[Kato]{Hisao Kato}
\email[Darji]{ubdarj01@louisville.edu}
\email[Kato]{hkato@math.tsukuba.ac.jp}
\address[Darji]{Department of Mathematics, University of Louisville, Louisville, KY 40292, USA.}
\address[Kato]{Institute of Mathematics, University of Tsukuba, Ibaraki 305-8571, Japan}
\thanks{The first author thanks the hospitality of the Institute of Mathematics, Tsukuba University where this research was conducted.}

\maketitle

\begin{abstract}We use recent developments in local entropy theory to prove that
chaos in dynamical systems implies the existence of complicated structure in the 
underlying space. 
Earlier Mouron proved that if $X$ is an arc-like continuum which admits a homeomorphism 
$f$ with positive topological entropy, then $X$ contains an indecomposable subcontinuum. 
Barge and Diamond proved that if $G$ is a finite graph and $f:G \rightarrow G$ is any map with positive topological entropy, then the inverse limit space $\varprojlim(G,f)$ contains an indecomposable continuum.
In this paper we show that if $X$ is a $G$-like continuum for some finite graph $G$ and $f:X \rightarrow X$
is any map with positive topological entropy, then $\varprojlim (X,f)$ contains an indecomposable continuum.  As a corollary, we obtain that in the case that $f$ is a homeomorphism, $X$ contains an indecomposable continuum. Moreover, if $f$ has uniformly positive upper entropy, then $X$ is an indecomposable continuum. Our results answer some questions raised by Mouron and generalize
the above mentioned work of Mouron and also that of Barge and Diamond. 
We also introduce a
new concept called zigzag pair which attempts to capture the complexity of
a dynamical systems from the continuum theoretic perspective and facilitates
the proof of the main result.
\end{abstract}

\section{Introduction} During the last thirty years or so, many interesting connections between topological dynamics and continuum theory have been established. One of the underlying themes is that somehow
complicated dynamics should imply existence of complicated continua. More precisely, if $X$ is a continuum (i.e., compact connected metric space) and $h:X \rightarrow X$ is a homeomorphism of $X$ such that $(X,h)$ is chaotic in some sense, then $X$ should contain a complicated subcontinuum. In the case that $h$ happens to be simply a continuous surjection, then
the inverse limit space $\varprojlim (X,h)$ should contain a complicated continuum. 

This line of investigation began with the seminal paper of Barge and Martin \cite{bm} where they showed that if $f$ is a continuous self-map of the
the interval $I = [0,1]$ which has a periodic point not a power of 2, then $\varprojlim(I,f)$ contains an indecomposable continuum.
By Misiurewicz's theorem \cite{bc}, we have that a continuous map of  the interval has a periodic point not a power of 2 if and only if the map has positive topological entropy.  Hence, we have that if $f: I \rightarrow I$ has positive topological entropy, then $\varprojlim(I,f)$ contains an indecomposable continuum. As one of the generally accepted definitions of chaos is that the map have positive topology entropy, we have, in the case of the interval, that chaos implies the existence of  a complicated subcontinuum in the inverse limit space. 

Let $G$ be a compact metric space. A continuous mapping $g$ from $X$ onto $G$ is an \emph{ $\eps$-mapping}
if for every $x \in X$, the diameter of $g^{-1}(x)$ is less than $\eps$. A continuum $X$ is \emph{ $G$-like} if for every $\eps>0$ there is a $\eps$-mapping from $X$ onto $G$. Or, equivalently, $X$ is homeomorphic to the inverse limit of a sequence of continuous surjections of $G$. (See Section~2 for definitions.) In particular, arc-like continua are those which are $G$-like for $G=[0,1]$.

In 1989, Ingram \cite{ingram} showed that
if $X$ is an arc-like continuum and $h$ is a homeomorphism of $X$ with a periodic point not a power of 2,
then $X$ contains an indecomposable continuum. In 1995, Ye
\cite{ye} generalized Ingram's result and that of Barge and Martin  by showing that if $X$ is arc-like and $f:X \rightarrow X$ is a continuous
surjection with a periodic point not a power of 2, then $\varprojlim(X,f)$ contains an indecomposable continuum.
As of now it is not known if positive entropy maps of arc-like hereditarily decomposable continua must have a periodic point which
is not a power of 2. Hence, one cannot conclude directly from Ye's result that positive entropy on arc-like
continua implies indecomposability in the corresponding inverse limit space. Ye in the same paper and also
in \cite{lxy2} proved that for certain special types of  homeomorphisms of the arc-like continua, positive entropy implies that the domain contains an
indecomposable continuum.
In 2011, Mouron \cite{mouron} settled the more general problem by proving that if $h$ is a homeomorphism of an arc-like continuum and $h$ has positive entropy, then $X$ contains an indecomposable continua. He raised the questions if the same result holds for
the monotone open maps of arc-like continua or, in general, for monotone maps of arc-like continua.

Dynamical systems where $X$ is more general than arc-like continua also have been investigated from this perspective. In 1990 Seidler \cite{seidler} showed that if $X$ is a continuum which admits a homeomorphism
of positive topological entropy, then $X$ is not a regular continuum. A continuum is \emph{ regular} if for
every $\eps >0$, there is an $n \in \N$ such that $X$ has at most $n$ pairwise disjoint continua of diameter
bigger than $\eps$. In 2004 Kato \cite{kato1} generalized Siedler's result by showing that one can replace 
"homeomorphism" in the above statement with a monotone map or a confluent map.  In \cite{kato2} a detailed
analysis of topological entropy of maps on regular curves was carried out. 

Barge and Diamond in 1994 \cite{bd} showed a remarkable result. They showed that for piecewise monotone
surjections of graphs, the conditions of having positive entropy, containing a horse shoe and the inverse limit space containing an indecomposable subcontinuum are all equivalent. Some modifications and extension of 
their results were given in \cite{lxy}.
However, one cannot hope to generalize their
result to arbitrary maps. A classical example of Henderson \cite{henderson} show that there is a map of the interval with no
chaotic behavior, in particular with topological entropy zero, whose inverse limit space is hereditarily indecomposable. Recently, Boro\'ski and Oprocha \cite{bo} have shown that one cannot relax the hypothesis of positive topological entropy to Li-Yorke chaos. Thus, the best one can hope for is to obtain results where positive entropy implies indecomposability in the corresponding inverse limit space. On the other hand, there are restrictions on the domain space also as observed by Mouron \cite{mouron}. There is a homeomorphism of the Cantor fan which has positive
topological entropy. \\

Our main theorem in this paper is the following:
\begin{repeatthm} Let $G$ be a finite graph, $X$ be a $G$-like continuum and $f:X \rightarrow X$ be a continuous
surjection with positive entropy. Then, $\varprojlim (X,f)$ contains an indecomposable continuum. In particular, if $f$ is a homeomorphism, then $X$ contains an indecomposable continuum.
\end{repeatthm}
This theorem generalizes that of Mouron \cite{mouron}. It also generalizes the above mentioned theorem of
Barge and Diamond \cite{bd}. Moreover, using some easily proved lemmas
and our main result, we can answer some questions left open in \cite{mouron}.

\begin{cornonum1} Let $G$ be a tree and $X$ be a $G$-like continuum. If $h: X \rightarrow X$ is a
continuous monotone surjection of $X$, then $X$ contains an indecomposable continuum. 
\end{cornonum1}

As in the article of Mouron \cite{mouron}, recent developments in local entropy theory are essential
to our proofs. We refer the reader to \cite{glasnerye} for a general survey on the subject. Inspired by the concepts of entropy
pairs and IE-pair, we introduce a new concept called \emph{ zigzag pair} (or Z-pair for short.)
This notion attempts to capture the complicated continuum theoretic behavior of a topological dynamical
system. The precise definition is somewhat involved and stated in the next section. Our main result is 
broken down into several pieces. The first piece show that if $X$ is a graph-like continuum for some fixed
graph $G$ and $h$ is a homeomorphism of $X$ with positive topological entropy, then $X$ has
a Z-pair. Then, in the case that $X$ is $G$-like for some tree $G$, we obtain that every irreducible
continuum between the Z-pair is indecomposable. In the general case where $X$ is $G$-like
for some graph $G$, we obtain that $X$ contains an indecomposable continuum but not necessarily
between every Z-pair. Their proofs have some features in common but are proved in separate sections.

The notion of uniformly positive entropy was introduced by Blanchard \cite{blanchard}. It has become
an useful and popular notion. We obtain the following corollary concerning this notion.

\begin{cornonum2} Let $G$ be a graph, $X$ a $G$-like continuum and $h:X \rightarrow X$ a homeomorphism
with u.p.e. Then, $X$ is an indecomposable continuum.
\end{cornonum2}

\section{Definitions and Notation}
A \emph{ continuum} is a compact connected metric space. We say that a continuum is \emph{nondegenerate}
if it has more than one point. A continuum is \emph{ indecomposable}
if it has more than one point and is not the union of two proper subcontinua.  \emph{ A continuum $X$ is unicoherent} if $H, K$ are subcontinua of $X$ such that $X = H \cup K$, then $H \cap K$ is a continuum.
It is \emph{ hereditarily unicoherent} if every subcontinuum of $X$ is unicoherent. In particular, $X$ is 
hereditarily unicoherent if and only if the intersection of any two subcontinua of $X$ is again a subcontinuum
of $X$.

Let $G$ be a compact metric space. A continuous mapping $g$ from $X$ onto $G$ is an \emph{ $\eps$-mapping}
if for every $x \in X$, the diameter of $g^{-1}(x)$ is less than $\eps$. A continuum $X$ is \emph{ $G$-like} if for every $\eps>0$ there is a $\eps$-mapping from $X$ onto $G$.  Our focus in this article is on $G$-like
continua where $G$ is a graph. In this article all graphs have
finitely many vertices. We use topological
and graph theoretic properties of graphs without explicitly stating so.

All maps are assumed to be continuous unless otherwise stated. Moreover, all self-maps
are assumed to be surjective unless otherwise stated.

Let $f:X \rightarrow Y$ be a map. We say that \emph{ $f$ is monotone} if $f^{-1}(x)$ is connected whenever nonempty. We recall that this definition is equivalent to stating that the preimage of a continuum
is a continuum whenever nonempty.

If $f:X \rightarrow X$ is a  map, then we use \emph{ $\varprojlim (X,f)$ to denote
the inverse limit of $X$ with $f$ as the bonding maps}, i.e., 
\[ \varprojlim (X,f) = \left \{ (x_i) \in X^{\N}: f(x_{i+1}) = x_i \right \}.
\]
The topology on $\varprojlim (X,f)$ is the subspace topology inherited from the product topology
on $X^{\N}$. In particular, if $X$ is compact, then $\varprojlim (X,f)$
is compact and, similarly, if $X$ is a continuum, then so is $\varprojlim (X,f)$.
The reader may refer to \cite{nadler}, \cite{kurt} for standard facts concerning 
continuum theory.\\

We will freely use the following collection of standard facts from continuum theory.
\begin{fact}\label{contfacts}
Let $X$ be a continuum and $f: X \rightarrow X$ be a  map.
\begin{enumerate}
\item \label{5} If $X$ is $G$-like, then so is $\varprojlim (X,f)$.
\item \label{2} If $X$ is $G$-like for a tree $G$, then $X$ is hereditarily unicoherent.
\item \label{4} If $X$ is hereditarily unicoherent, then so is $\varprojlim (X,f)$. 
\item \label{7} If $f$ is monotone, then the projection of $\varprojlim (X,f)$ onto the
$n^{th}$ coordinate is also a monotone map. 
\item \label{1}If $X$ is hereditarily unicoherent and $f$ is monotone, then the restriction of $f$
to any continuum is also a monotone map. 
\item \label{6.5} Monotone continuous image of an indecomposable continuum is also
indecomposable provided that the image has more than one point.
\end{enumerate}
\end{fact}

If ${\mathcal G}$ is a collection of sets, then the \emph{ nerve of ${\mathcal G}$} is the graph whose
vertices are elements of ${\mathcal G}$ and there is an edge between two distinct vertices $g_1, g_2$
if $g_1 \cap g_2 \neq \emptyset$. (We do not allow edges from a vertex to itself.)   If $\{C_1, \ldots, C_n\}$ is a subcollection of ${\mathcal G}$ we call 
it a \emph{chain} if $C_i \cap C_{i+1} \neq \emptyset$ for $1\le i <n$ and
$\overline{C_i }\cap \overline{ C_j }\neq \emptyset $ implies that $|i - j| \le 1$. We say
that $\{C_1, \ldots, C_n\}$ is a \emph{ free chain in ${\mathcal G}$} if it is a chain and, moreover,
for all $1 < i <n$ we have that $C \in {\mathcal G}$ with $\overline{C} \cap \overline{C_i} \neq \emptyset$ implies
that $C=C_{i-1}$ or $C=C_{i+1}$.

A \emph{ topological dynamical system}  is an ordered pair $(X, f)$ where $X$ is a compact metric space and $f:X \rightarrow X$
is a surjective map. In this article, we only consider dynamical systems where $X$ is connected. Implications
of having positive topological entropy is the focus of our investigation. As in \cite{mouron} we make an extensive use of local entropy theory. 
Below are the definitions of relevant concepts. The reader can find more details in
\cite{glasnerye} and \cite{kerrli}.

Let $X$ be a compact metric space and ${\mathcal U}, {\mathcal V}$ be two covers of $X$. Then,
\[ {\mathcal U} \vee {\mathcal V} = \{ U \cap V: U \in {\mathcal U}, V \in {\mathcal V}\}.
\]
The quantity $N({\mathcal U})$  denote minimal cardinality of a subcollection of ${\mathcal U}$ which
covers $X$. Let $f:X \rightarrow X$ be continuous and ${\mathcal U}$ be an open cover of $X$. Then,
\[h_{top} ({\mathcal U},f) = \limsup_{n \rightarrow \infty} \frac{ \ln  \left [N({\mathcal U} \vee
f^{-1}( {\mathcal U}) \vee \ldots \vee f^{-n+1}( {\mathcal U}) )  \right ]}{n}.
\]
The \emph{ topological entropy of $f$,} denoted by $h_{top}(f)$, is simply the supremum of $h_{top} (f, {\mathcal U})$
over all open covers ${\mathcal U}$ of $X$.

The following standard fact is a classical result of Bowen \cite{bowen}.
\begin{fact}\label{bowen} Let $f: X \rightarrow X$ and $\sigma$ the shift map on $\varprojlim (X,f)$ defined by 
$\sigma(x_1,x_2,\ldots) = (f(x_1), x_1, x_2,x_3,\ldots)$. Then, $\sigma$ is a homeomorphism
with entropy that of $f$.

\end{fact}

Let $X$ be a compact metric space and $f:X \rightarrow X$ be a homeomorphism. Let ${\mathcal A}$
be a collection of subsets of $X$. We say that \emph{  ${\mathcal A}$ has an independence set
with positive density  (for brevity, $A$ has an i.s.p.d.)}  if there exists a set $I \subset \N$  with positive density such that for all finite sets $J \subseteq I$,
and for all $(Y_j) \in \prod _{j\in J}{\mathcal A}$, we have that
$\cap_{j\in J}f^{-j}(Y_j)\neq \emptyset$. We recall that set $I \subseteq \N$ has positive density if
$\lim_{n \rightarrow \infty} \frac{|I \cap [1,n]|}{n} >0$. We observe a simple but important and and useful
fact that if $I$ is an i.s.p.d. for ${\mathcal A}$ then for all $k \in \Z$, $k+I$ is an i.s.p.d for ${\mathcal A}$.

We now recall the definition of IE-tuple. Let $(x_1, \ldots, x_n)$ be a sequence of points in $X$.
We say that \emph{ $(x_1, \ldots, x_n)$ is a IE-tuple for $f$} if whenever $A_1, \ldots ,A_n$ are open sets
containing $x_1, \ldots, x_n$, respectively, we have that the collection ${\mathcal A} =\{A_1,
\ldots, A_n\}$ has an independence set with positive density.  In the case that $n=2$, we use
the term IE-pair.  (We drop the phrase ``for $f$"  and simply say IE-tuple or IE-pair when the mapping is clear from the context.) 
We use \emph{ $IE_k$} to denote the set of all IE-tuples of length $k$.
The diagonal of $X^k$ is denoted by $\Delta_k(X)$ or $\Delta_k$ if the underlying space
is clear from the context.
We will freely use the following facts from the local entropy theory. 
\begin{fact}\label{entfacts} \cite{kerrli} [Propositions, 3.8, 3.9]
Let $X$ be a compact metric space and $f: X \rightarrow X$ be a map.
\begin{enumerate}
\item Let $(A_1, \ldots, A_k)$ be a tuple  of closed subsets of $X$ which has an independent set
of positive density. Then, there is and IE-tuple $(x_1, \ldots, x_k)$ with $x_i \in A_i$
for $1 \le i \le k$.
\item $h_{top}(f) >0$ if and only if $f$ has an IE-pair $(x_1, x_2)$ with $x_1 \neq x_2$.
\item $IE_k$ is closed and $f \times \ldots \times f$ invariant subset of $X^k$.
\item If $(A_1, \ldots, A_k)$ has an i.s.p.d. and, for $1 \le i \le k$,
${\mathcal A}_i$ is a finite collection of sets such that $A_i \subseteq \cup{\mathcal A}_i$, 
then there is $A'_i \in {\mathcal A}_i$ such that $(A'_1, \ldots, A'_k)$ has an i.s.p.d. 
\end{enumerate}
\end{fact}

Blanchard \cite{blanchard} introduced the notion of \emph{uniform positive entropy (u.p.e)}.
We will obtain some stronger consequences when homeomorphism $h: X \rightarrow X$ has
the additional property of having u.p.e. As our definition of u.p.e, we use the following equivalent
formulation given in \cite{kerrli}: $h: X \rightarrow X$ has
u.p.e. if and only if every tuple of $X^2$ is an IE-tuple for $f$.

Inspired by the notion of IE-pair, we introduce concepts of $l$-zigzag and Z-pair which will facilitate our proofs.
Let $X$ be a continuum which is $G$-like for some graph $G$, $f: X \rightarrow X$, $U, V$ two subsets of $X$ and ${\mathcal G}$ an open cover of $X$. Let $l >1$ be odd. We say that a \emph{ chain $\{C_1, \ldots, C_n\} \subseteq {\mathcal G}$ is a
$l$-zigzag from $U$ to $V$} if there exists $1=k_1 < k_2 <\ldots k_{l+1} =n$ such that
\begin{itemize}
\item for all $i$ odd, $C_{k_i} \cap U\neq \emptyset$, 
\item for all $i$ even, $C_{k_i} \cap  V \neq \emptyset$, and
\item $\{C_{k_i}\cap U: 1 \le i  \le l+1, i \text { odd} \} \bigcup \{C_{k_i}\cap V: 1 \le i  \le l+1, i \text{ even} \}$ has an i.s.p.d. 
\end{itemize}

A pair \emph{ $(x, y)$ is a Z-pair for $f$} if for every open sets $U, V$, contains $x,y$ and for
every $\eps >0$ and for all odd $l \in \N$, we have that there is an $\eps$-cover ${\mathcal G}$ of $X$ whose
nerve is $G$
and a \textbf{free} chain $\{C_1, \ldots, C_n\} \subseteq {\mathcal G}$, with $x \in C_1, y \in C_n$, which is a
$l$-zigzag from $U$ to $V$. (We drop the phrase ``for $f$" and say simply Z-pair when the mapping is clear from the context.) We use $Z(X)$ to denote the set of Z-pairs subset of $X$.

\section{Z-pairs}
In this section we state some background and preparatory results. We also show that we can obtain
a Z-pair arbitrary close an IE-pair.  

\begin{prop}\label{sumset}Let $I \subseteq \N$ be a set with positive density and $n \in \N$. Then,
there is a finite set $F \subseteq I$ with $|F|=n$ and a positive density set $B$ such that
$F+ B \subseteq I$.
\end{prop}
\begin{proof} A proof this can be obtain in a manner of Theorem 9 of \cite{mouron}. However, here we give a simple proof of this pointed out to us by Luca Zamboni. For each $x \in I$, let $I_x= I -x$.
Then, each $I_x$ has the same density as $I$. Then, $\{I_x: x \in I \}$ is an infinite collection of 
sets with the same density. Then, by the basic theory of densities, we can find $x_1, \ldots, x_n \in I$
such that $\cap _{i=1}^n I_{x_n}$ has positive lower density. Let $B \subseteq \cap _{i=1}^n I_{x_n}$
a set with positive density. Then, for each $1 \le i \le n$, we have that \[x_i+B \subseteq x_i +I_{x_i}
= x_i + (I -x_i)  =I.\]
Then, $B$ and $F=\{x_1, \ldots, x_n\}$ are the desired sets.
\end{proof}
\begin{prop}\label{enlargeindsets} Let $X$ be a compact metric space and $f:X \rightarrow X$.
Let ${\mathcal A}$ be a collection which has an i.s.p.d. and $n \in \N$. Then, there is a 
finite set $F$ with $|F|=n$ such that
\[{\mathcal A}_F = \{ \cap_{i \in F} f^{-1}(Y_i): Y_i \in {\mathcal A}\}
\]
has an i.s.p.d.
\end{prop}
\begin{proof} Let $I$ be an i.s.p.d for ${\mathcal A}$. Let $F, B$ be as in Lemma~\ref{sumset}. We claim
that $B$ is an independent set for ${\mathcal A}_F$.  Indeed, let $J \subseteq B$ be any finite set
and $Z_j \in {\mathcal A}_F$ for $j \in J$. By definition $Z_j = \cap_{i \in F} f^{-i}(Y^j_i)$
where $Y^j_i \in {\mathcal A}$. 
Then,
\[\cap_{j \in J}f^{-j}(Z_j) = \cap_{j \in J}f^{-j}(\cap_{i \in F} f^{-i}(Y^j_i)) = \cap_{j \in J}\cap_{i \in F} f^{-(i+j)}(Y^j_i).
\]
As $F+J \subseteq I$, $Y^j_i \in {\mathcal A}$, and $I$ an independent set for ${\mathcal A}$,
we have that  
\[\cap_{j \in J}\cap_{i \in F} f^{-(i+j)}(Y^j_i) \neq \emptyset,
\]
completing the proof.
\end{proof}
In the following proposition, if $\sigma \in \{0,1\}^n$, we write $\sigma=(\sigma(1),\sigma(2),...\sigma(n))$,  where $\sigma(i) \in \{0,1\}$. 
\begin{prop}\label{combo}Let $ l, n \ge 1$,  and $\sigma_1, \ldots, \sigma_{(n+2)(n+1)^{l-1}} $ be distinct elements of $\{0,1\}^n$.
Then, there are $i, \ 1\le k_1<k_2 < k_3 <  \ldots < k_{2^l} \le {(n+2)(n+1)^{l-1}}$ such that 
$\sigma_{k_j}(i) =0$ for $j$ odd and $\sigma_{k_j}(i) =1$ for $j$ even.
\end{prop}
\begin{proof} 
Let us first consider the case when $l=1$. For each $1\le j \le n+1 $, let 
$1 \le i_j  \le n$ be such that $\sigma_j(i_j) \neq \sigma
_{j+1}(i_j)$. This is possible since all the elements of the sequence are distinct. Now,
by the pigeonhole principle, there are $1 \le m < m' \le n+1$ such that $i_m = i_{m'}$. 
Let $i=i_m$.
Note that we have that  $m < m+1 \le m' < m'+1$ and $\sigma_m(i) \neq \sigma_{m+1}(i)$
and $\sigma_{m'}(i) \neq \sigma_{m'+1}(i)$. Since all $\sigma$'s take values of $0$ or $1$,
there is $k_1 \in \{m, m+1\}$ such that $\sigma_{k_1}(i)=0$ and there is $k_2 \in \{m', m'+1\}$
such that $\sigma_{k_2}(i) =1$.

To prove the general case, recall the elementary fact that if we have a function from a set of 
$(n+1)^k$ objects into $n$ objects, then at least $2^{k}$ many objects must map to the same value.
(This can be easily proved by the pigeonhole principle and induction.) Now, in the general case,
we decompose our sequence of length $(n+2)(n+1)^{l-1}$ into $(n+1)^{l-1}$ blocks, each
block consisting of  $(n+2)$ consecutive terms from the original sequence.  For each block we obtain $1 \le i \le n$ and $k< k'$ which satisfy the condition that $\sigma_k(i) =0$ and $\sigma_{k'}(i)=1$. Then, we must have $2^{l-1}$
many blocks which must satisfy our theorem for the same value of $i$. As each block contains
$k<k'$ with $\sigma_{k}(i)=0$ and $\sigma_{k'}(i)=1$, the proof is complete.
\end{proof}
In the following proposition, ${\mathcal A}_F$ is as defined in Proposition~\ref{enlargeindsets}.
\begin{lem}\label{zigzagprep} Let $X$ be a continuum, $f: X \rightarrow X$ be a homeomorphism, 
${\mathcal A}= \{A_0,A_1\}$ be subsets of $X$, $F \subseteq \N$ with $|F| =n$.
Furthermore, assume that ${\mathcal C}$ is a chain consisting of open subsets of $X$ such
that each element of ${\mathcal C}$ intersects at most one element
of ${\mathcal A}_F$ and that there is a subcollection ${\mathcal C}'$
of ${\mathcal C}$ of cardinality  at least $(n+2)(n+1)^{l-1}$ such that $\{L \cap (\cup 
{\mathcal A}_F): L \in {\mathcal C}'\}$ has an i.s.p.d.
Then, there is a subchain
${\mathcal D}$ of ${\mathcal C}$ and an $i \in F$ such that $f^i({\mathcal D})$
is a $(2^l-1)$-zigzag chain from $A_0$ to $A_1$.
\end{lem}
\begin{proof} Let ${\mathcal C}= \{C_1, \ldots, C_k\}$
and ${\mathcal C}' = \{C_{m_1}, C_{ m_2}, \ldots, C_{m_t}\}$.
As each $C_{m_j}$ intersects  exactly one element
of ${\mathcal A}_F$, we may chose $\sigma_j \in \{0,1\}^F$ such that $C_{m_j} \cap 
\left ( \cap_{s \in F}f^{-s}(A_{\sigma_j(s)})  \right )\neq \emptyset$. By hypothesis
we have that $\{\sigma_1, \ldots, \sigma_t\}$ is a sequence of distinct elements of 
$\{0,1\}^F$ whose length is at least $(n+2) (n+1)^{l-1}$. Hence, by Proposition~\ref{combo}
we have that there is $i \in F$ and a sequence $ 1 \le p_1 < p_2 \ldots < p_{2^l}$
such that $\sigma_{p_j}(i)  =0$ for $j$ odd and $\sigma_{p_j} (i) =1$ for $j$ even.
Now let ${\mathcal D}$ be the subchain of ${\mathcal C}$ defined by
$\{C_{m_{p_1}} , C_{m_{p_1}+1}, \ldots , C_{m_{p_{2^l}}}\}$. 

Let us now observe that ${\mathcal D}$ has the desired property. Let $j$ be odd. By
definition of $\sigma$'s, we have that $ C_{m_{p_j}} \cap f^{-i}(A_{\sigma_{p_j}(i)}) \neq 
\emptyset$. As $\sigma_{p_j}(i) =0$, we have that $f ^i (C_{m_{p_j}}) \cap A_0 \neq
\emptyset$, for $j$ odd. A similar argument show that for $j$ even, $f ^i (C_{m_{p_j}}) \cap A_1 \neq 
\emptyset$. 

That $\{ f ^i(C_{m_{p_j}}) \cap A_0: j \mbox{ odd }, 1 \le j \le 2^l\} \cup \{ f ^i(C_{m_{p_j}}) \cap A_1: j \mbox{ even}, 1 \le j \le 2^l\}$ has an i.s.p.d follows from the fact that 
$\{L \cap (\cup 
{\mathcal A}_F): L \in {\mathcal C}'\}$ has an i.s.p.d and that the homeomorphic image of a 
collection with an i.s.p.d is again a collection which has an i.s.p.d.
\end{proof}
\begin{lem}\label{zigzag} Let $X$ be a continuum which is $G$-like for some graph $G$, 
$f:X \rightarrow X$ a homeomorphism, $\eps >0$ and $l >1$ odd.  If $(A_0,A_1)$ is
a disjoint pair of closed sets which has an i.s.p.d., then there is a $\eps$-cover ${\mathcal H}$ of $X$ whose nerve is $G$
and a free chain ${\mathcal E } \subseteq {\mathcal H}$ such that ${\mathcal E}$ is an $l$-zigzag from $A_0$ to 
$A_1$. 
\end{lem}
\begin{proof}Let $|G|$ be the number of edges of $G$. Choose $n$ large enough so that $2^n > |G| (n+2)(n+1)^{l-1}$. Let $I$ be an independent set  of positive density for $(A_0,A_1)$. Applying Lemma~\ref{sumset}, we obtain $F$
and $B$ as in the conclusion of that lemma. As $A_0,A_1$ are disjoint closed sets which has an i.s.p.d.,
we have that ${\mathcal A}_F$ is a pairwise disjoint collection of nonempty closed sets of cardinality $2^n$. 
Moreover, ${\mathcal A}_F$ has an i.s.p.d., namely $B$. Let $\delta >0$ be small enough so that $\delta < \eps$ and every pair of distinct elements of ${\mathcal A}_F$ are at least $\delta$ apart.  
Moreover, by the continuity of $f$, we may choose $\delta$  sufficiently small so that that
if $U \subseteq X$ has diameter less than $\delta$, then $f^i(U)$ has diameter less than $\eps$ for all
$i \in F$. As $X$ is $G$-like, we may choose an $\delta$-cover ${\mathcal G}$ of $X$ whose nerve is $G$.  

Note that each element of ${\mathcal G}$ intersects at most one element of ${\mathcal A}_F$. 
Applying Fact~\ref{entfacts} (4), we obtain a subcollection ${\mathcal G'}$ of ${\mathcal G}$ of 
cardinality $2^n$ such that $\{ g' \cap  \left (\cup {\mathcal A}_F \right ) : g' \in {\mathcal G'}\}$  has an i.s.p.d. Moreover, there exists a free chain in ${\mathcal G}$, corresponding to an edge of $G$, such that
this free chain intersects at least $2^n/|G| $ many element of ${\mathcal G'}$. Call this
free chain ${\mathcal C}$ and ${\mathcal C'}$ those element of ${\mathcal G'}$ belonging to ${\mathcal C}$. As $2^n/ |G| > (n+2)(n+1)^{l-1}$, we have that all the hypothesis of 
Lemma~\ref{zigzagprep} are satisfied. Hence by Lemma~\ref{zigzagprep}, we have that there is a subchain
${\mathcal D}$ of ${\mathcal C}$ and an $i \in F$ such that $f^i({\mathcal D})$
is a $(2^l-1)$-zigzag chain from $A_0$ to $A_1$. Let ${\mathcal H}=f^i({\mathcal G})$ and ${\mathcal E}= f^i({\mathcal D})$. 
\end{proof} 
\begin{lem}\label{zpair}
Let $X$ be a continuum which is $G$-like for some graph $G$, 
$f:X \rightarrow X$ a homeomorphism, and $(A_0,A_1)$
a disjoint pair of closed sets which has an i.s.p.d. Then, there is $x \in A_0$ and $y \in A_1$ such that
$(x,y)$ is a Z-pair.
\end{lem}
\begin{proof}This lemma simply follows from a repeated application of Lemma~\ref{zigzag}.
\end{proof}
\begin{cor}\label{zpairdense} 
Let $X$ be a continuum which is $G$-like for some graph $G$
and $f:X \rightarrow X$ a homeomorphism with positive entropy. Then, arbitrarily close to every $(x, y) \in IE_2(X) \setminus \Delta_2(X)$, there is a Z-pair. 
\end{cor}
\begin{proof} Let $(x, y) \in IE_2(X) \setminus \Delta_2(X)$ and let $A_0, A_1$ be two disjoint closed
sets which are the closures of open sets containing $x,y$, respectively. Then, by the definition of IE-pair, we have that $\{A_0,A_1\}$ has an i.s.p.d. Now the proof follows from Lemma~\ref{zpair}. 
\end{proof}
\section{Tree-like case}
In this section we prove that in the tree-like case, the existence of Z-pair implies indecomposability between every irreducible continuum between the Z-pair. As a corollary we obtain that positive entropy implies indecomposability. We need some definitions and a result due to Krasinkiewicz and Minc. 

Let $U$ and $V$ be nonempty disjoint open subsets of a continuum $X$ and let 
$A$ and $B$ be two disjoint closed subsets of $X$. Then \emph{  
$(X,A,B)$ is crooked in $X$ } between $U$ and $V$ (see \cite{mk}) 
if $A\subset U$, 
$B\subset V$ and there exist three closed subsets $F_0, F_1, F_2$ of $X$ such that \\
(i)~$X=F_0\cup F_1 \cup F_2$,\\
(ii)~$F_0\cap F_2=\phi$,\\
(iii)~$A\subset F_0, B\subset F_2$, and \\
(iv)~ $F_0\cap F_1\subset V, F_1\cap F_2\subset U$.\\  
We will use the following theorem of Krasinkiewicz and Minc.
\begin{thmnonum}[Krasinkiewicz-Minc, \cite{mk}] Let $A, B$ be two nonempty disjoint closed subsets of continuum $X$.
Then, the following are equivalent.
\begin{enumerate}
\item $(X, A, B)$ is crooked between every neighborhood of $A$ and every neighborhood of $B$.
\item Every irreducible continuum between $A$ and $B$ is indecomposable.
\end{enumerate}
\end{thmnonum}
\begin{thm}\label{zpairind} Let $X$ be a $G$-like continuum for a tree $G$,  and $f:X \rightarrow X$ a homeomorphism with a Z-pair $(a,b)$. Then, every
irreducible continuum between $a$ and $b$ is indecomposable. In particular, $X$ contains an indecomposable continuum containing $a$ and $b$.
\end{thm}
\begin{proof} In light of the Theorem of Krasinkiewicz-Minc, it suffices to show that $(X,a,b)$ is crooked between
an $\eps$ ball $U$ centered at $a$  and an $\eps$ ball $V$ centered at $b$. We may assume that $\eps$
small enough so that $d(a,b) < \frac{\eps}{8}$.  Let $U'$ and $V'$ be balls of radius 
$\eps/ 4$ centered at $a,b$, respectively. 
By the definition of Z-pair, there is an open $\eps/ 4$-cover ${\mathcal G}$ of $X$ whose nerve is $G$, a free chain ${\mathcal C} = \{C_1, \ldots, C_n\}$ in ${\mathcal G}$ which is a 3-zigzag between $U'$ and $V'$ with $a \in C_1, b \in C_n$. Let $1 = k_1 < k_2 \ldots <k_4=n$ be the integers which witness 3-zigzag. As the nerve of 
${\mathcal G}$ is a tree and ${\mathcal C}$ is a free arc of ${\mathcal G}$, we have that ${\mathcal G} \setminus {\mathcal C}$ is the union of two disjoint collections whose nerves are trees, one whose union intersects only $C_1$ and the other whose union intersects only $C_n$. Let us call these collection them ${\mathcal G}_a, {\mathcal G}_b$, respectively. Moreover, as the nerve of  ${\mathcal G}$ is a tree, we have that $\overline{(\cup {\mathcal G}_a) } \cap \overline{ (\cup {\mathcal G}_b)} = \emptyset$. We 
now define the desired closed sets $F_0, F_1, F_2$.

\[ F_0 = \overline{ \cup{\mathcal G}_a \cup   \left ( \cup \{C_1 ,\ldots, C_{k_2}\} \right ) }
\]
\[ F_1 = \overline{  \cup \{C_ {k_2},\ldots, C_{k_3}\} }
\]
\[ F_2 = \overline{ \cup{\mathcal G}_b \cup   \left ( \cup \{C_{k_3} ,\ldots, C_n\} \right ) }
\]
It is clear that $F_0, F_1, F_2$ satisfy all the conditions of crookedness between $U,V$ , except possibly (iv). 
Let us show that $F_0 \cap F_1 \subseteq V$. Note that $F_0 \cap F_1 \subseteq \overline {C}_{k_2}$. 
As the diameter of $C_{k_2}$ is less  than $\eps/ 4$,  and it intersects $V'$, a ball centered at $b$ with
diameter $\eps/4$, we have that $\overline{C_{k_2}}$ is a subset of ball centered at $b$ with diameter
less than $\eps$. Hence, $F_0 \cap F_1 \subseteq V$. An analogous argument shows that $F_1 \cap F_2 \subseteq U$,
completing the proof.
\end{proof}
\begin{cor}\label{hominid} Let $X$ be  a $G$-like continuum for some tree $G$ and $h:X \rightarrow X$ a homeomorphism with positive entropy. Then $X$ contains an indecomposable continuum.
\end{cor} 
\begin{proof} Indeed, by Fact~\ref{entfacts} (2) we have that $X$ has an IE-par, by Theorem~\ref{zpairdense} we have that $X$ has a Z-pair and by Theorem~\ref{zpairind} we have that $X$ contains an indecomposable continuum.
\end{proof}
\begin{thm}\label{treecont}Let $X$ be a $G$-like continuum for some tree $G$ and $h:X \rightarrow X$ be a map with positive entropy. Then, $\varprojlim (X,f)$ contains an indecomposable continuum. 
\end{thm}
\begin{proof} Consider $\varprojlim (X,f)$ and $\sigma$ the shift map on $\varprojlim (X,f)$. 
By Fact \ref{contfacts}(\ref{5})  we have that $\varprojlim (X,f)$ is $G$-like and by Fact (\ref{bowen})  $\sigma$ is a homeomorphism
of  $\varprojlim (X,f)$ with positive entropy. Hence, by 
Corollary~\ref{hominid} we have that   $\varprojlim (X,f)$ contains an indecomposable continuum.
\end{proof}
Using the following lemma, we can conclude, in the case that $f:X \rightarrow X$ is a monotone 
map with positive entropy, the existence of indecomposable subcontinuum persists.
The corollary following the lemma below answers questions of Mouron \cite{mouron}.
\begin{lem}\label{lem:repeat}Let $X$ be a continuum and $f$ a continuous, monotone map of $X$.
If $X$ is hereditarily unicoherent and $\varprojlim (X,f)$ contains an indecomposable continuum,
then so does $X$.
\end{lem}
\begin{proof}
By Fact~\ref{contfacts}(\ref{4}) we have that $\varprojlim (X,f)$ is hereditarily unicoherent and 
by Fact~\ref{contfacts}(\ref{7}) $\pi_n$, the $n^{th}$ projection,
is a monotone map for all $n$. Let $M$ be an indecomposable subcontinuum of $\varprojlim (X,f)$.
Choose $n$ large enough so that $\pi_n(M)$ has more than one point. By Fact~\ref{contfacts}(\ref{1}),
$\pi_n|M$ is also monotone. By By Fact~\ref{contfacts}(\ref{6.5}), we have that $\pi_n(M)$ is 
indecomposable, completing the proof.
\end{proof}
\begin{cor}\label{cor:repeat1} Let $G$ be a tree and $X$ a $G$-like continuum. If $f: X \rightarrow X$ is a
monotone map of $X$ with positive entropy, then $X$ contains an indecomposable continuum. 
\end{cor}
\begin{proof} By Theorem~\ref{treecont}, we have that $\varprojlim (X,f)$ contains an indecomposable
continuum. As $X$ is $G$-like for a tree $G$, we have, by Fact~\ref{contfacts}(\ref{2}), that $X$ is hereditarily unicoherent.   Now by Lemma~\ref{lem:repeat} we have that $X$ contains an indecomposable continuum.
\end{proof}
\section{General Graph Case}
In this section we consider the general graph case. Before we proceed, we need to introduce some notation.

Let ${\mathcal G}$ and ${\mathcal H}$ be collections of subsets of $X$. By \emph{ ${\mathcal G} [{\mathcal H}]$}
we mean the collection $\{g \in {\mathcal G}: \exists h \in {\mathcal H} \text { with } \overline{h} \subseteq g\}.$

Let $X$ be a continuum and $\{{\mathcal G}_n\}$ be a sequence of covers of $X$. We say
that {\em $\{{\mathcal G}_n\}$ is a defining sequence of $X$ } provided that the following conditions hold:
\begin{itemize}
\item ${\mathcal G}_{n+1}$ is a refinement of ${\mathcal G}_n$, i.e., for each $g \in {\mathcal G}_{n+1}$, there is $g' \in {\mathcal G}_n$ such that $\overline {g} \subseteq g'$, and 
\item $\lim_{n \rightarrow } mesh ({\mathcal G}_n) = 0$. 

\end{itemize}
\begin{lem}\label{indcriteria} Let $X$ be a continuum and $\{{\mathcal G}_n\}$ be a defining sequence of $X$. Furthermore, assume that for each $n$, there exists a free chain ${\mathcal C}_n \subseteq {\mathcal G}_n$  and disjoint subchains  of ${\mathcal D}_n $,  and ${\mathcal E}_n$ of ${\mathcal C}_n$ and 
such
that ${\mathcal C}_n[{\mathcal D}_{n+1}] ={\mathcal C}_n[{\mathcal E}_{n+1}] ={\mathcal C}_n$.
Then, $X$ contains an indecomposable continuum.
\end{lem}
\begin{proof} By passing through a subsequence if necessary, we may assume that $\{\overline{\cup {\mathcal C}_n}\}_{n=1}^{\infty}$ converges in Hausdorff metric to a compact set $M \subseteq X$.
As each ${\mathcal C}_n$ is a chain and $\lim_{n \rightarrow \infty} mesh ({\mathcal G}_n) = 0$, we have that $M$ is a continuum. We next show that $M$ is indecomposable. Let $H$ be a proper subcontinuum of $M$. It will suffice to show that $H$ is nowhere dense in $M$. To this end, let $p \in H$ and $\eps >0$. We will show there is $q \in M \setminus H$ such that the distance between $p$ and $q$ is less than
$\eps$.  

We first observe that if $C \in {\mathcal C}_n$ for some $n$, then ${C} \cap M \neq \emptyset$. Indeed, this is the case as each $C \in {\mathcal C}_n$ has the property that there is $C' \in {\mathcal C}_{n+1}$ such that $\overline {C'} \subseteq C$. 

Choose $n$ sufficiently large so that the following conditions hold.
\begin{itemize}
\item $d_H(M, \overline{\cup {\mathcal C}_n}) < \frac{\eps}{4}$,
\item there exists $U \in {\mathcal C}_n$ such that $U \cap H = \emptyset$, and
\item $mesh ({\mathcal C}_n) < \frac{\eps}{4}$.
\end{itemize}
The first condition above is possible as $M$ is the Hausdorff limit of $\{\overline{\cup {\mathcal C}_n}\}_{n=1}^{\infty}$. The second condition is possible as $M \setminus H \neq \emptyset$.
Let $C \in {\mathcal C}_n$ so that $d_H(p, C) < \frac{\eps}{2}$.

We next observe that the free chain ${\mathcal C}_{n+2}$ has four pairwise disjoint subchains
${\mathcal A}_i$, $ 1 \le i \le 4$, such that ${\mathcal C}_n[{\mathcal A}_{i}] = {\mathcal C}_n$.
Let ${\mathcal H}$
be those elements of ${\mathcal G}_{n+2}$ which intersects $H$. As $H$ is continuum and 
${\mathcal C}_{n+2}$ is a free chain of ${\mathcal G}_{n+2}$, we have that there is an
$j \in \{1, 2, 3, 4\}$ such that ${\mathcal A}_j \subseteq {\mathcal H}$ or ${\mathcal A}_j
\cap {\mathcal H} = \emptyset$. We next observe that ${\mathcal A}_j \subseteq {\mathcal H}$
cannot happen. For otherwise, as ${\mathcal C}_n[{\mathcal A}_{j}] = {\mathcal C}_n$,
there is $a \in {\mathcal A}_j$ such that $a \subseteq U$. As $a \in {\mathcal H}$, we have
that $a \cap H \neq \emptyset$, $a \subseteq U$ but $U \cap H = \emptyset$, leading to a contradiction.
Hence, we have that ${\mathcal A}_j \cap {\mathcal H} = \emptyset$. Now, let $C' \in {\mathcal A}_j$
be such that $\overline{C'} \subseteq C$. Let $q \in C' \cap M$. Then, $q \in M\setminus H$. Moreover,
as $d_H(p, C) < \frac{\eps}{2}$, we have that $d(p,q) < \eps$, completing the proof of nowhere denseness of $H$.
\end{proof}
\begin{lem}\label{genhelp}Let $G$ be a finite graph, $X$ be $G$-like, $h:X \rightarrow X$ be a homeomorphism and $U, V$ two disjoint nonempty opens subsets of $X$. Furthermore, assume that ${\mathcal G}$ is an open cover of $X$, ${\mathcal G}'= \{g_1, g_2, \ldots, g_n\}$ a free chain of ${\mathcal G}$ and $1 < a_1 < b_1 <a_2 < b_2 \ldots < a_6 < b_6 < n$ are such that
\begin{itemize}
\item $\overline{g}_{a_i} \subseteq U$ for all $1 \le i \le 6$,
\item $\overline{g}_{b_i} \subseteq V$ for all $ 1 \le 6$, and 
\item $ \{ g_{a_i} , g_{b_i}: 1 \le i \le 6\}$ has an independent set of positive density.
\end{itemize}
Then, there exits an open cover ${\mathcal H}$ which is a refinement of ${\mathcal G}$, a free chain  ${\mathcal H}'=\{h_1, h_2, \ldots, h_m\}$ of ${\mathcal H}$ and $1 < c_1 < d_1 <c_2 < d_2 \ldots < c_6 < d_6 < m$  such that
\begin{enumerate}
\item $\overline{h}_{c_i} \subseteq U$ for all $1 \le i \le 6$,
\item $\overline{h}_{d_i} \subseteq V$ for all $ 1 \le 6$, 
\item $ \{ h_{c_i} , h_{d_i}: 1 \le i \le 6\}$ has an independent set of positive density, and 
\item for one of ${\mathcal C} =\{g_{a_2}, \ldots , g_{b_3} \}$ or ${\mathcal C} =\{g_{a_4}, \ldots , g_{b_6} \}$ the following holds.
\[ {\mathcal C} [{\mathcal A}_i] = {\mathcal C}_i,  \ \ \ \forall \ 1 \le i \le 6,\] 
where ${\mathcal A}_i$ is the subchain $\{ h_{c_i}, \ldots, h_{d_i}\}$ of ${\mathcal H}'$. 
\end{enumerate}
\end{lem}
\begin{proof} By Fact~\ref{entfacts} (1) we have that there is an IE-pair one of whose 
coordinates belong to $\overline{g}_{a_1}$ and the other one to $\overline{g}_{b_4}$.
By Corollary~\ref{zpairdense}, we have that there is a Z-pair such that one of its coordinates belong to $ g_s \cap U$, $|s-a_1| \le 1$, and the other belongs to $g_t \cap V$, $|t-b_4| \le 1$.  Using the definition of Z-pair we can easily obtain an open cover ${\mathcal H}$ which is a refinement of ${\mathcal G}$, a free chain  ${\mathcal H}'=\{h_1, h_2, \ldots, h_m\}$ of ${\mathcal H}$ which is 23-zigzag between $g_s \cap U$ and $g_t \cap V$. That is, there is $1 < m_1 < m_2 \ldots m_{24} <m$ such that
\begin{itemize}
\item $\overline{h}_{m_i} \subseteq g_s \cap U$, for $i$ odd,
\item $\overline{h}_{m_i} \subseteq g_t  \cap V$, for $i$ even, and
\item $\{h_{m_i}:1 \le i \le 24\}$ has i.s.p.d.
\end{itemize}
 As ${\mathcal G'}$ is a free chain of ${\mathcal G}$ we have that for one of  ${\mathcal C} =\{g_{a_2}, \ldots , g_{b_3} \}$ or ${\mathcal C} =\{g_{a_4}, \ldots , g_{b_6} \}$, there are at least 6 
 odd integers $i$ such that the subchain ${\mathcal A} = \{h_{m_i}, \ldots h_{m_{i+1}}\}$ of ${\mathcal H}'$ has the property that 
 \[ \mathcal{C}[{\mathcal A}]=\mathcal{C},
\]
completing the proof of the lemma.
\end{proof}
\begin{lem}\label{genind} Let $G$ be a finite graph, $X$ be $G$-like and $h:X \rightarrow X$ be a homeomorphism. Suppose that $(x,y)$ is an IE-pair of distinct points and $U, V$ are open sets containing $x,y$, respectively. Then, $X$ contains an indecomposable continuum which intersects $U$ and $V$.
\end{lem}
\begin{proof}By Corollary~\ref{zpairdense} and repeated applications of Lemma~\ref{genhelp}, we may choose a defining sequence ${\mathcal G}_n$ of $X$, a free chain ${\mathcal G}'_n$ of ${\mathcal G}_n$,  six pairwise disjoint ${\mathcal A}_{n,i}$, $1 \le i \le 6$, subchains of ${\mathcal G}'_n$,  and  ${\mathcal C}_n$ which satisfies the conclusion of Lemma~\ref{genhelp}.
Note that ${\mathcal C}_n$ is the union of  two disjoint ${\mathcal A}_{n,i}$'s. Call one of the ${\mathcal D}_i$ and the other ${\mathcal E}_i$. Then,
$\{({\mathcal C}_i, {\mathcal D}_i, {\mathcal E}_i)\}$ satisfy the hypothesis of Lemma~\ref{indcriteria}. Hence $X$ contains an indecomposable continuum which intersect $U$ and $V$.
\end{proof}
\begin{cor}\label{genhomind}
Let $G$ be a finite graph, $X$ be $G$-like and $h:X \rightarrow X$ be a homeomorphism
with positive entropy.  Then, $X$ contains an indecomposable continuum. 
\end{cor}
\begin{thm}\label{thm:repeat} Let $G$ be a finite graph, $X$ be a $G$-like continuum and $f:X \rightarrow X$ be a continuous
surjection with positive entropy. Then, $\varprojlim (X,f)$ contains an indecomposable continuum. 
\end{thm}
\begin{proof} Proof of this theorem is analogous to that of Theorem~\ref{treecont}. 
Consider $\varprojlim (X,f)$ and $\sigma$ the shift map on $\varprojlim (X,f)$. 
By Fact \ref{contfacts}(\ref{5})  we have that $\varprojlim (X,f)$ is $G$-like and by Fact (\ref{bowen})  $\sigma$ is a homeomorphism
of  $\varprojlim (X,f)$ with positive entropy. Hence, by 
Corollary~\ref{genhomind} we have that   $\varprojlim (X,f)$ contains an indecomposable continuum.

\end{proof}

\begin{cor}\label{cor:repeat2} Let $G$ be a graph, $X$ a $G$-like continuum and $h:X \rightarrow X$ a homeomorphism
with u.p.e. Then, $X$ is an indecomposable continuum.
\end{cor}
\begin{proof}To obtain a contradiction, assume that $X$ is not indecomposable and let $H ,K$
be two proper subcontinua of $X$ such that $H \cup K = X$. Let $U = X \setminus H$ and $V = X \setminus K$. As $h$ has u.p.e., by Corollary~\ref{zpairdense} we know theres $a \in U$ and $b \in V$ such that $(a,b)$ is a Z-pair. Now, using the definition of Z-pair, we may choose an open cover ${\mathcal G}$, a free chain ${\mathcal C} = \{C_1, \ldots, C_n\}$ of ${\mathcal G}$ and $1 =m_1 < m_2 < \ldots < m_8=n$  such that $C_{m_i} \subseteq U$
for $i$ odd and $C_{m_i} \subseteq V$ for $i$ even. For each $1 \le i \le 7$, let ${\mathcal A}_i = \{C_{m_i}, \ldots, C_{m_{i+1}}\}$.  As $H$ is connected, $U = X \setminus H$ and ${\mathcal C}$ is a free chain of ${\mathcal G}$, we have that $H$ can intersects ${\mathcal A_i}$ for at most
three $i$'s in $\{1, \ldots, 7\}$. Similarly, $K$ can intersects ${\mathcal A_i}$ for at most
three $i$'s in $\{1, \ldots, 7\}$. Hence for some $1 \le i \le 7$, we have that ${\mathcal A}_i$ intersects neither $H$ nor $K$. However, this is a contradiction as $X =  H \cup K$. Hence, $X$ is indecomposable. 

\end{proof}
\section{Final Remarks and Problems}
If $f:X \rightarrow X$ and $X$ is an interval, then it is well-known that the following conditions are equivalent
\cite{bc}.
\begin{displaymath} 
\xymatrix{ f \textit{ has positive topological entropy}. \ar@2{<->}[d]\\
f \textit { has a periodic point not a power of 2}.\ar@2{<->}[d]\\
f^n \textit{ has a horseshoe for some } n \in \N.}
\end{displaymath}
By  $f^n$ has a horseshoe, we mean that there are two disjoint closed intervals $A, B$ subset of $X$
such that $A \cup B \subseteq f^n(A) \cap f^n(B)$. Using the fact that $f^n$ has a horseshoe, one can easily
show that $\varprojlim(X,f)$ contains an indecomposable continuum as was done in \cite{bm}. It is also
easy to deduce the periodicity condition from the existence of a horseshoe. Hence, the condition of having
a horseshoe is an important condition.

In the case of $X$ an arc-like hereditarily decomposable continua, what is known is that having periodic points not a power of 2 implies positive entropy \cite{mendez}. Also, for a restricted class of functions and continua,
positive entropy implies the existence of  periodic points not a power of 2 \cite{lxy2}. We conjecture that the above equivalent conditions diagram holds
for all $X$ which are arc-like and hereditarily decomposable with the obvious modification in the definition of horseshoe (i.e., $A, B$ are disjoint continua instead of intervals.) Again, from the existence of horseshoe, one can easily conclude the existence of indecomposable continua in the inverse limit. 

By a beautiful paper of Llibre and Misiurewicz \cite{llibre} we have the following equivalent conditions for 
$f:G \rightarrow G$, $G$ a finite graph.
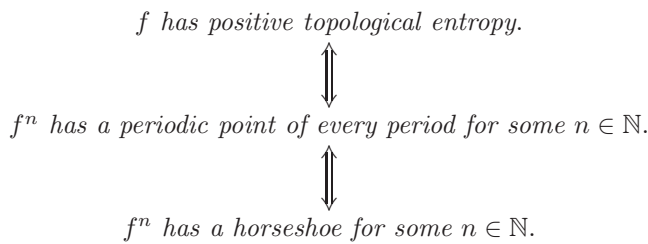
\begin{displaymath} 
\xymatrix{ f \textit{ has positive topological entropy}. \ar@2{<->}[d]\\
f ^n\textit { has a periodic point of every period for some } n \in \N. \ar@2{<->}[d]\\
f^n \textit{ has a horseshoe for some } n \in \N.}
\end{displaymath}
Again the definition of horseshoe has to be modified appropriately \cite{llibre}. (The reader may wonder why discrepancy in the second condition. This arises from the fact that Sharkovskii's theorem does not hold in graphs. However, by Sharkovskii's theorem on the interval and  by the
Minc-Transue's generalization of Sharkovskii's theorem on arc-like continua \cite{mt}, we have that if $f$ has a periodic point not a power of 2, then for some $n \in \N$, $f^n$ has a periodic point of every period.) We conjecture that the above equivalent conditions diagram holds
for all $X$ hereditarily decomposable and $G$-like for some graph $G$. If this conjecture holds, then our main results in this paper can be deduced from the existence of a horseshoe.

\end{document}